\documentclass[11pt,reqno,tbtags]{amsart}

\usepackage[usenames,dvipsnames]{xcolor}
\RequirePackage[T1]{fontenc}

\usepackage[colorlinks,citecolor=blue,urlcolor=blue,linkcolor=blue,linktocpage=true]{hyperref}

\usepackage[]{amsmath}
\usepackage{amssymb,amsthm,amsfonts,amsbsy,latexsym}

\usepackage{graphicx}

\usepackage{appendix}
\usepackage{url}

\usepackage[usenames,dvipsnames]{xcolor}            
\usepackage{graphicx}
\usepackage{bbm,bm}
\usepackage{comment}
\usepackage{mathtools}
\usepackage{placeins}
\usepackage[labelfont={sl}]{caption}
\usepackage[labelfont={normalfont}]{subcaption}

\usepackage{amssymb,amsfonts,amsmath,amsthm,amscd,dsfont,mathrsfs}

\usepackage{url}

\usepackage{enumerate}
\usepackage{hyperref}

\numberwithin{equation}{section}

\newtheorem{theorem}{Theorem}[section]

\newtheorem{lemma}[theorem]{Lemma}
\newtheorem{corollary}[theorem]{Corollary}

\theoremstyle{definition}
\newtheorem{definition}[theorem]{Definition}
\newtheorem{example}[theorem]{Example}
\newtheorem{remark}[theorem]{Remark}
\newtheorem{condition}[theorem]{Condition}

\newcommand{\E}{\mathbf{E}}
\newcommand{\R}{\mathbb{R}}
\newcommand{\Sphere}{\mathbb{S}^{d-1}}
\newcommand{\bv}{\mathbf{V}}

\newcommand{\N}{\mathbb{N}}
\newcommand{\sM}{\mathcal{M}}
\newcommand{\sN}{\mathcal{N}}
\newcommand{\fS}{\mathfrak{S}}
\newcommand{\eps}{\varepsilon}
\renewcommand{\P}{\mathbf{P}}
\newcommand{\Prob}[1]{\mathbf P\{#1\}}
\newcommand{\vd}{\xrightarrow{d}}

\usepackage{color}
\usepackage[numbers, sort&compress]{natbib} 

\begin{document}

\title[Convergence to Scale-Invariant Poisson Processes]{Convergence to Scale-Invariant Poisson Processes and Applications in Dickman Approximation}
\author{Chinmoy Bhattacharjee}
\address{IMSV, University of Bern}
\email{chinmoy.bhattacharjee@stat.unibe.ch}

\author{Ilya Molchanov}
\address{IMSV, University of Bern}
\email{ilya.molchanov@stat.unibe.ch}

\date{\today}

\thanks{Supported by the Swiss National Science Foundation
Grant No.\ 200021\_175584.}


\begin{abstract}
  We study weak convergence of a sequence of point processes to a
  scale-invariant simple point process. For a deterministic sequence
  $(z_n)_{n\in\mathbb{N}}$ of positive real numbers increasing to infinity
  as $n \to \infty$ and a sequence $(X_k)_{k\in\mathbb{N}}$ of independent
  non-negative integer-valued random variables, we consider the
  sequence of point processes
  \begin{equation*}
  \nu_n=\sum_{k=1}^\infty X_k \delta_{z_k/z_n}, \quad n \in \mathbb{N},
  \end{equation*}
  and prove that, under some general conditions, it converges
  vaguely in distribution to a scale-invariant Poisson process
  $\eta_c$ on $(0,\infty)$ with the intensity measure having the
  density $ct^{-1}$, $t\in(0,\infty)$. An important motivating example
  from probabilistic number theory relies on choosing $X_k \sim {\rm
  	Geom}(1-1/p_k)$ and $z_k=\log p_k$, $k \in  \mathbb{N}$, where $(p_k)_{k \in
  	\mathbb{N}}$ is an enumeration of the primes in increasing order. We
  derive a general result on convergence of the integrals $\int_0^1
  t \nu_n(dt)$ to the integral $\int_0^1 t \eta_c(dt)$, the latter having a
  generalized Dickman distribution, thus providing a new way of
  proving Dickman convergence results.
  
  We extend our results to the multivariate setting and provide
  sufficient conditions for vague convergence in distribution for a
  broad class of sequences of point processes obtained by mapping the points from $(0,\infty)$ to $\mathbb{R}^d$ via multiplication by i.i.d.\ random vectors. In
  addition, we introduce a new class of multivariate Dickman
  distributions which naturally extends the univariate setting.
\end{abstract}

\maketitle

\section{Introduction}
\label{sec:introduction}

Consider a locally compact separable metric space $S$ with Borel
$\sigma$-algebra $\fS$.  Let $\sM(S)$ denote the space of all locally
finite non-negative measures on $S$. This space is endowed with the
\textit{vague} topology generated by assuming continuity of the
integration maps $\mu \mapsto \mu f=\int_S f(x) \mu(dx)$ for all $f$
from the family $\widehat{C}_S$ of bounded non-negative continuous
functions on $S$ with relatively compact support.  A \emph{random measure}
$\xi$ is a random element in $\sM(S)$, equivalently, $\xi
A=\xi\mathds{1}_A$ is a random variable for each relatively compact
Borel set $A$. The associated notion of convergence in distribution of
random measures is called \emph{vague convergence} in distribution,
denoted hereafter by $\vd$, see \cite{dal:ver03,dal:ver08}. When
considering \emph{point processes}, we restrict ourselves to the subclass
$\sN(S) \subset \sM(S)$ of counting measures (that
is, taking values in $\N_0$, the set of non-negative integers). A random measure $\xi$ is said to have a finite intensity if $\E  (\xi A)<\infty$ for every relatively compact Borel set $A$. 

In this paper, we are particularly interested in vague convergence in
distribution to scale-invariant Poisson processes.  A random measure
$\xi$ on $S$ is \emph{scale-invariant} if its distribution is invariant with
respect to a group of scaling transformations of $S$. 
Even though convergence to
stationary Poisson processes has been extensively studied in the
literature, studies regarding convergence to scale-invariant processes
seem to be rare. Distributional properties of scale-invariant
Poisson processes on the half-line $(0,\infty)$ are surveyed in
\cite{arr98}. While a simple transformation relates a scale-invariant
Poisson process on $(0,\infty)$ to a stationary Poisson processes on
the line, such a transformation is not readily available in general
Euclidean spaces.

Throughout the sequel, we take $S=\R^d\setminus \{0\}$, $d \in \N$, that is, the Euclidean space with the origin removed. On the half-line, for $c>0$, we denote by $\eta_c$
the scale-invariant Poisson process on $(0,\infty)$ with intensity
measure $ct^{-1}dt$, and we will simply write $\eta$ for $\eta_1$.

Scale-invariant processes naturally arise as limits of point processes
when a scaling is applied to the support points of the point
processes. For measures, this amounts to scaling of their arguments,
namely, the scaling of $\nu\in\sM(S)$ by $t>0$ is defined as 
\begin{equation}
\label{eq:IS}
T_t \nu (A)=\nu(t^{-1}A), \quad A \in \fS.
\end{equation}
We call this operation \textit{intrinsic scaling}. In Section~\ref{sec:intsc}, we show that random measures when intrinsically scaled, naturally yield scale-invariant measures as limits. As an application, we generalize a result in \cite{cov09} proving that the intrinsically scaled process of jump sizes in a pure-jump subordinator converges vaguely in distribution to a scale-invariant Poisson process, and as a consequence, the sum of small jumps in the process converges to a Dickman distribution.

In this paper, our basic objects of interest are point processes on
$(0,\infty)$ of the following type. Let $(z_k)_{k\in\N}$ be a
sequence of positive \emph{deterministic} numbers with $z_n \uparrow \infty$
as $n \to \infty$. For a sequence $(X_k)_{k\in\N}$ of independent
random variables in $\N_0$, define the point process
\begin{equation*}
\nu=\sum_{k=1}^\infty X_k \delta_{z_k},
\end{equation*}
where $\delta_x$ denotes the Dirac measure at $x$. Rescaling
the support points of $\nu$ by $(z_n)_{n \in \N}$ yields the sequence of point processes
\begin{equation}
\label{eq:nun}
\nu_{n}A=T_{z_n}\nu( A)=\nu (z_n^{-1}A), \quad A \in \fS, n\in \N.
\end{equation}
In Section~\ref{sec:conv-scale-invar}, we study the convergence of
such processes; these results are extended to point processes in
multidimensional Euclidean spaces in Section~\ref{sec.hd}.

Our interest in the scale-invariant Poisson process $\eta_c$ also
stems from its connection to the \emph{Dickman distributions}. It is well known that the sum of points of $\eta_c$ lying in the interval $(0,1)$
is distributed as a generalized Dickman random variable denoted
hereafter by $D_c$ for $c>0$, with $D=D_1$ being a standard Dickman
random variable. The generalized Dickman distribution with parameter
$c>0$ can be defined as the unique non-negative fixed point of the
distributional transformation $W \mapsto W^*$ given by
\begin{equation*}
W^*=_d Q^{1/c}(W+1),
\end{equation*}
where $=_d$ denotes equality in distribution and $Q$ is a uniformly
distributed random variable on $[0,1]$ independent of $W$. It was
introduced in the work of Dickman \cite{Di30} in the context of
smooth numbers and since then has appeared, sometimes
curiously, in various areas including probabilistic number theory
\cite{bh17,Pi16}, minimal directed spanning trees \cite{BR04,PW04},
quickselect sorting algorithm \cite{Go17,HT02} and log-combinatorial
structures \cite{ABT,BaNi11}.

Given the various application, not
surprisingly, there have been many works studying weak convergence to Dickman distributions \cite{HT02,PW04,Pi16}
and, more recently, Stein's method has been used to provide
non-asymptotic bounds for Dickman approximations
\cite{ar16,bh17,Go17}. In \cite{Pi18}, Pinsky provided some
general conditions under which certain randomly weighted Bernoulli
sums converge to a generalized Dickman random variable. But, to the
best of our knowledge, there has been no other attempt to characterize the domain of attraction of the Dickman
distributions. Elaborating on \cite{RA:PC}, one aim of this work is to identify a broad class of
random variables which asymptotically behave like a Dickman random
variable. To do this, we make use of the fact that
\begin{displaymath}
D_c=_d \int_0^1 t\eta_c(dt)=\sum_{t \in \eta_c \cap (0,1)}t.
\end{displaymath}
Hence, if a sequence of point processes converges
vaguely in distribution to $\eta_c$, then, under certain natural
additional conditions, sums of their points in the interval $(0,1)$
converge in distribution to the Dickman random variable $D_c$. Thus, our approach via
scale-invariant Poisson processes yields a new tool to prove Dickman
convergences and provides useful insights into why such convergences
occur. We note here that a similar approach concerning limit theorems for point processes in relation to the behaviour of sums of their points has previously been discussed in \cite{AGK16}. Also, the simpler case of Poisson processes converging to $\eta_c$ on $(0,\infty)$ was considered in
\cite{cov09}. Scale-invariant Poisson processes also arise in 
limit theorems for records, see e.g.\ \cite{bas:plan:soul18} and
references therein.  

In Section~\ref{sec:scale-invar-high}, we characterize scale-invariant
Poisson processes in general dimension $d$, and show that any such
process can be obtained by independently multiplying each point of a
scale-invariant Poisson process on $(0,\infty)$ with independent and identically
distributed unit vectors in $\R^d$. Such a characterization naturally leads to a
multivariate generalization of the Dickman distribution. Analogous to the univariate case, these multivariate Dickman distributions are fixed points of a distributional transform 
\begin{displaymath}
W^*=_d Q^{1/c}(W+U),
\end{displaymath}
where $Q$ is a uniform random variable on $[0,1]$ and $U$ a unit random vector in $\R^d$, independent of everything else.

Some results concerning weak convergence of general point processes
(not necessarily scale-invariant) are collected in the
Appendix.

\section{Intrinsic scaling of random measures}\label{sec:intsc}

Let $\widehat \fS \subseteq \fS$ denote the family of relatively compact Borel sets in
$S=\R^d \setminus\{0\}$ for some $d \in \N$. A subclass $\mathcal{U}
\subset \widehat \fS$ is called \textit{dissecting} if every open set can
be expressed as a countable union of sets from $\mathcal{U}$ and every
set in $\widehat \fS$ can be covered by finitely many sets in $
\mathcal{U}$.  Recall that a subclass $\mathcal{I} \subset \widehat \fS$
is a \textit{ring} if it is closed under proper differences and under
finite unions and intersections. In the special case of $(0,\infty)$,
we will often take the dissecting ring $\mathcal{U}$ to be the family of
finite unions of semi-open intervals $(a,b]$ with $0<a<b<\infty$.

Let $(\xi_n)_{n\in\N}$ be a sequence of point processes in $S$.  It is
well known that the vague convergence in distribution $\xi_n \vd \xi$
for a simple $\xi$ follows from the one-dimensional weak convergences $\xi_n A
\vd \xi A$ for all $A$ from the dissecting ring 
\begin{displaymath}
\mathcal{U} \subset \widehat \fS_{\E \xi}
=\{B \in \widehat \fS: \E\; \xi (\partial B)=0\},
\end{displaymath}
where $\partial B$ denotes the boundary of $B$, see e.g.\
\cite[Chapter 4]{kalle17}. A measure $\mu \in \mathcal{M}(S)$ is said to be \emph{scale-invariant} if $T_c \mu= \mu$ for all $c >0$, where $T_c$ is defined at \eqref{eq:IS}. The next result shows that the limit of the sequence of random
measures obtained by intrinsic scalings of
a given random measure $\nu$ is necessarily scale-invariant under some
mild conditions on the normalizing constants. For deterministic
measures, similar results are known, see e.g.\
\cite[Theorem~3.1]{lin:res:roy14}. We write $\Sphere$ for the $d$-dimensional unit sphere and $B_r$ for the closed ball of radius $r>0$ around the origin.

\begin{lemma}\label{lem:1}
	Let $(s_n)_{n\in\N}$ be a sequence of positive real numbers
	increasing to infinity with $\lim_{n \to \infty}s_{n-1}/s_n= 1$, and let
	$\mu, \nu \in \mathcal{M}(S)$ be random measures with finite intensities such that $T_{s_n} \nu \vd
	\mu$ as $n \to \infty$. Then
	$T_t \nu \vd \mu$ as $t \to \infty$,
	and the limiting measure $\mu$ is scale-invariant.
\end{lemma}

\begin{proof}
	Since $\mu$ has finite intensity, the family of sets 
	\begin{align*}
	\mathcal{U}=\Big\{A \times[a,b] :  \E\; \mu[\partial A \times(0,\infty)]=&\E\; \mu[\partial (B_a) \cup \partial(B_b)]=0, \\
	 &A \subseteq \Sphere, 0<a<b<\infty \Big\}
	\end{align*}
	forms a dissecting semi-ring. Hence, the first claim will follow (see
	\cite[Theoreme~1.1]{K73}) by establishing that 
	\begin{equation}
	\label{eq:SIConv}
	\left(T_t \nu(A_i \times[a_i,b_i])\right)_{i \in [k]}
	\vd \left(\mu(A_i \times[a_i,b_i])\right)_{i \in [k]}
	\quad\text{as }\; n \to \infty
	\end{equation}
	for all $k\in\N$ and $A_i \times[a_i,b_i] \in \mathcal{U}$, $i=1,\dots,k$.  
	
	To simplify the argument, assume that $k=1$; for general $k \in \N$, one can argue
	similarly.  For $t>0$, let $n(t)$ be the integer such that $s_{n(t)}<t
	\le s_{n(t)+1}$. Fix a Borel set $A \subseteq \Sphere$ and $0<a<b<\infty$ with $A \times [a,b] \in \mathcal{U}$
	and $\eps \in (0,b-a)$. Since $\lim_{n \to \infty}s_{n-1}/s_n= 1$ and $n(t)
	\to \infty$ as $t \to \infty$,
	\begin{displaymath}
	\frac{a}{s_{n(t)+1}} > \frac{a-\eps}{s_{n(t)}} \quad
	\text{and} \quad
	\frac{b}{s_{n(t)+1}}
	>\frac{b-\eps}{s_{n(t)}}
	\end{displaymath}
	for all sufficiently large $t$.  Hence, for $t$ large enough, we have
	\begin{displaymath}
	T_t \nu(A \times[a,b]) \le \nu(A \times [a/s_{n(t)+1},b/s_{n(t)}])
	\le T_{s_{n(t)}} \nu(A \times [a-\eps,b]).
	\end{displaymath}
	A similar argument yields a lower bound, so that 
	\begin{displaymath}
	T_{s_{n(t)}} \nu(A \times [a,b-\eps])
	\le T_t \nu(A \times[a,b])\le T_{s_{n(t)}} \nu(A \times [a-\eps,b])
	\end{displaymath}
	for all sufficiently large $t$.  Since $n(t) \to \infty$ as $t \to
	\infty$ and $T_{s_n} \nu \vd \mu$ as $n \to \infty$, we obtain that
	\begin{displaymath}
	\limsup_{t \to \infty} \Prob{T_t \nu(A \times [a,b]) \le x}
	\le \Prob{\mu(A \times [a,b-\eps]) \le x}
	\end{displaymath}
	and
	\begin{displaymath}
	\liminf_{t \to \infty} \Prob{T_t \nu(A \times [a,b]) \le x}
	\ge \Prob{\mu(A \times [a-\eps,b]) \le x}
	\end{displaymath}
	for $x\ge 0$.  Since $\E\; \mu[\partial (B_a) \cup \partial (B_b)]=0$, 
	\[
	\lim_{\eps \to 0}\Prob{\mu(A \times [a,b-\eps]) \le x}=\lim_{\eps \to 0}\Prob{\mu(A \times [a-\eps,b]) \le x}
	=\Prob{\mu(A \times [a,b]) \le x},
	\]
	which, together with the two inequalities above yield \eqref{eq:SIConv}, proving the first claim.
	
	Finally, let $v:S \to \R$ be a bounded continuous function with relatively compact support. For $c>0$, since $T_t \nu \vd \mu$ as $t \to \infty$,
	\begin{align*}
	\lim_{t \to \infty}T_c T_t \nu (v)&=\lim_{t \to \infty}\int_S v(x)T_{ct}\nu(dx)\\
	&=\lim_{t \to \infty}\int_S v(cx)T_t \nu(dx)=\int_S v(cx)\mu(dx)=T_c \mu(v),
	\end{align*}
	which implies that
	\[
	T_c T_t \nu \vd T_c \mu \quad \text{as }\; t \to \infty.
	\]
	On the other hand, $T_c T_t \nu=T_{ct} \nu$ converges vaguely in distribution to $\mu$ as $t \to \infty$ by our assumption. Hence we obtain $T_c \mu= \mu$, proving the scale invariance of $\mu$. 
\end{proof}

The following theorem proves Dickman convergence for the sums of
small jump sizes in a pure-jump subordinator; we note here that the
Dickman limit result is not new and has been proved in
\cite{cov09}. We prove a stronger result that the scaled point
process of jump sizes converges to a scale-invariant Poisson process
on $(0,\infty)$.

Let $Y=(Y(t))_{t \ge 0}$ be a pure-jump subordinator with infinite
L\'evy measure $\sigma$ and for $\eps>0$, let $Y_\eps$ be the process obtained by removing the jumps of size larger than $\eps$ in the L\'evy-Ito decomposition of $Y$. For $t>0$, let $\Pi_t$ denote the point process of jump sizes occurring in the time interval $[0,t]$. The scaled process $T_{1/\eps}\Pi_t$ consists of the points of $\Pi_t$ scaled by $\eps$. Recall, $D_c$ denotes a Dickman distributed random variable with parameter $c>0$.

\begin{theorem}\label{thm:covo}
	If $\eps^{-1}\int_0^\eps x \sigma(dx) \to c>0$ as $\eps \to
	0$, then for any $t > 0$,
	\begin{displaymath}
	T_{1/\eps}\Pi_t \vd \eta_{ct} \quad \text{as }\; \eps \to 0.
	\end{displaymath}
	Moreover,
	\begin{displaymath}
	\eps^{-1}Y_{\eps}(t) \vd D_{ct} \quad \text{as }\; \eps \to 0.
	\end{displaymath}
\end{theorem}

\begin{proof}
	Arguing as in the proof of \cite[Theorem~2.1]{cov09}, letting $\psi$
	and $\psi_\eps$, $\eps>0$ be the measures given by
	$\psi(dx)=\mathds{1}_{(0,1]}(x)c dx$ and $\psi_\eps(dx)= x
	\cdot T_{1/\eps} \sigma(dx)=x \sigma(\eps dx)$ respectively, for
	any $p \in (0,1)$, we have
	\begin{displaymath}
	\psi_\eps((0,p])=\int_0^p x \sigma(\eps dx)
	= \frac{1}{\eps} \int_0^{p\eps} z \sigma(dz) \to cp
	=\psi((0,p])\quad \text{as } \;\eps \to 0.
	\end{displaymath}
	By Lemma~\ref{conv.discont}, 
	\begin{displaymath}
	T_{1/\eps} \sigma((p,1])=\int_p^1 x^{-1} \psi_\eps(dx)
	\to \int_p^1 x^{-1}\psi(dx)=c\log (1/p) \quad \text{as }\; \eps \to 0,
	\end{displaymath}
	which yields that the Poisson process on $(0,\infty)$ with intensity measure $T_{1/\eps} \sigma$ converges vaguely in distribution to $\eta_c$ as
	$\eps \to 0$. Since $Y$ is a L{\'e}vy process with L{\'e}vy
	measure $\sigma$, the jump process $\Pi_t$ is distributed as a Poisson process on $(0,\infty)$ with intensity measure
	$t\sigma$; this proves the first claim.
	
	Finally, note that $\eps^{-1}Y_{\eps}(t)=\int_0^1 x\;
	(T_{1/\eps}\Pi_t)(dx)$. To prove the last claim, by
	Lemma~\ref{lem:gensum}, it suffices to check that
	\begin{equation}
	\label{eq:cov}
	\lim_{\delta \to 0}\limsup_{\eps \to 0}\E \int_0^\delta x\; (T_{1/\eps}\Pi_t)(dx)=0.
	\end{equation}
	Since $\Pi_t$ is a Poisson process with intensity measure
	$t\sigma$, we have that $T_{1/\eps}\Pi_t$ is distributed as a Poisson process on $(0,\infty)$ with intensity measure $tT_{1/\eps} \sigma$. Thus, using the Mecke equation in the first equality
	and that $\eps^{-1}\int_0^\eps x \sigma(dx) \to c$ as $\eps
	\to 0$ in the third, we obtain
	\begin{align*}
	\limsup_{\eps \to 0}\E \int_0^\delta x\;
	(T_{1/\eps}\Pi_t)(dx)
	&=\limsup_{\eps \to 0}t \int_0^\delta x T_{1/\eps} \sigma(dx)\\
	&=\limsup_{\eps \to 0}t\eps^{-1} \int_0^{\eps\delta} x \sigma(dx) =ct\delta
	\end{align*}
	which implies \eqref{eq:cov}, concluding the proof.
\end{proof}

\section{Convergence to scale-invariant Poisson processes}
\label{sec:conv-scale-invar}

Now we move our attention to proving convergence to scale-invariant
Poisson processes for sequences of general (not necessarily Poisson)
point processes. The necessary and sufficient conditions for vague
convergence in distribution of point processes to a simple point process given by
Theorem~\ref{thm1}, when applied to $\nu_n$ given by \eqref{eq:nun}
with $\eta_c$ being the limit, translate to the following simpler
condition. For convenience, denote 
\begin{displaymath}
q_k^0=\Prob{X_k=0} \quad \text{and}\quad q_k^1=\Prob{X_k=1}, \; k \ge 1.
\end{displaymath}

\begin{condition}\label{cond.1d}
	There exists $c>0$ such that for all $0<a<b<\infty$,
	\begin{enumerate}[(i)]
		\item $\prod_{k: a z_n < z_k \le b z_n} q_k^0 \to \left(a/b\right)^c$ as $n \to
		\infty$.
		\item $\liminf_{ n \to \infty }
		\sum_{k: a z_n < z_k \le b z_n}q_k^1/q_k^0\geq c\log (b/a)$.
	\end{enumerate}
\end{condition}

\begin{theorem}\label{thm2}
	A sequence of point processes $(\nu_{n})_{n\in\N}$ given by
	\eqref{eq:nun} converges vaguely in distribution to $\eta_c$ for
	some $c>0$ as $n \to \infty$ if and only if $(q_k^0,q_k^1)_{k \in
		\N}$ and $(z_n)_{n\in\N}$ satisfy Condition~\ref{cond.1d}.
\end{theorem}
\begin{proof}
	Condition~\ref{cond.1d}(i) for the dissecting ring composed of
	finite unions of semi-open intervals is equivalent to condition (i)
	in Theorem~\ref{thm1}. Condition (ii) in Theorem~\ref{thm1} is
	equivalent to
	\begin{equation*}
	\liminf_{ n \to \infty }
	\left[\left(1+ \sum_{k: az_n < z_k \le bz_n}q_k^1/q_k^0\right)\prod_{l: az_n <
		z_l \le bz_n} q_l^0   \right]
	\geq \left(\frac{a}{b}\right)^c\left(1+ c\log \frac{b}{a}\right),
	\end{equation*}
	which, given Condition~\ref{cond.1d}(i), simplifies to Condition~\ref{cond.1d}(ii), proving the result.
\end{proof}

The next result concerns vague convergence to scale-invariant Poisson processes for a large class of point processes $\nu_n$ of the form \eqref{eq:nun} and, as a consequence, establishes weak convergence of sums of the points in
$(0,1)$ of $\nu_{n}$ to a generalized Dickman distributed random
variable $D_c$. Note that such a convergence does not readily follow
from the vague convergence since $\eta_c$ has infinitely many points
in any neighbourhood of zero.

\begin{theorem}\label{ex}
	For a monotone sequence of positive numbers $(z_k)_{k \ge 0}$ increasing to
	infinity with $\lim_{k \to \infty}z_k/z_{k-1}=1$, let
	$(X_k)_{k\in\N}$ be independent random variables in $\N_0$ with
	\begin{displaymath}
	q_k^0=(z_{k-1}/z_k)^c \quad  \text{and} \quad
	q_k^1=q_k^0(1-q_k^0)
	\end{displaymath}
	for some $c>0$. Then the sequence $(\nu_{n})_{n\in\N}$ defined at
	\eqref{eq:nun} converges vaguely in distribution to $\eta_c$ as $n
	\to \infty$.  If, in addition, $\E X_k=\mathcal{O}(q_k^1)$, then
	\begin{equation}
	\label{Dickman}
	\frac{1}{z_n} \sum_{k=1}^n z_k X_k
	\vd  D_c \; \text{ as $n \to \infty$}.
	\end{equation}
\end{theorem}

\begin{proof}
	Fix $0<a<b<\infty$. Let $M=\inf\{k:az_n < z_k \le bz_n\}$ and
	$N=\sup\{k:az_n < z_k \le bz_n\}$. Letting $\delta_{n}=az_n - z_{M-1}$
	and $\delta'_{n}=bz_n-z_{N}$, one has
	\begin{displaymath}
	\frac{z_{M-1}}{z_{N}}=\frac{a-\delta_{n}/z_n}{b-\delta'_{n}/z_n}.
	\end{displaymath}
	Since $\lim_{k \to \infty}z_k/z_{k+1}=1$ and $M\to\infty$ as $n \to
	\infty$,
	\begin{displaymath}
	\limsup_{n \to \infty} \frac{\delta_{n}}{z_n}
	\le \lim_{n \to \infty} \frac{z_M-z_{M-1}}{z_M}\cdot \frac{z_M}{z_n} = 0,
	\end{displaymath}
	and a similar argument shows that $\limsup_{n \to \infty}
	\delta'_{n}/z_n=0$. Thus,
	\begin{displaymath}
	\prod_{k: az_n < z_k \le bz_n} q_k^0
	=\prod_{k: az_n < z_k \le bz_n}\left(\frac{z_{k-1}}{z_k}\right)^c
	=\left(\frac{z_{M-1}}{z_N}\right)^c \to \left(\frac{a}{b}\right)^c
	\quad \text{as }\; n \to\infty.
	\end{displaymath}
	Also,
	\begin{align*}
	\liminf_{ n \to\infty } \sum_{k: az_n < z_k \le bz_n}\frac{q_k^1}{q_k^0}
	&\ge \liminf _ { n \to\infty } \left(\frac{z_{n-1}}{z_n}\right)^c
	\liminf _ { n \to\infty }\sum_{k: az_n < z_k \le bz_n}\frac{z_k^c-z_{k-1}^c}{z_{k-1}^c}\\
	&\ge \liminf_{ n \to \infty }\int_{z_{M-1}^c}^{z_{N}^c}\frac{1}{t}dt
	=c \liminf_{ n \to\infty}\log \frac{z_{N}}{z_{M-1}}= c\log \frac{b}{a}.
	\end{align*}
	Hence, Condition~\ref{cond.1d} is satisfied and the first claim
	follows by Theorem~\ref{thm2}.
	
	If $\E X_k=\mathcal{O}(q_k^1)$, then there exists $C>0$ such that
	$\E X_k \le Cq_k^1$ for all $k\in\N$. Denoting by $\lceil \cdot
	\rceil$ the ceiling function and using the simple inequality that
	$1-(1-x)^c \le 2^{\lceil c \rceil}x$ for $x \in [0,1]$ in the
	penultimate step, we have
	\begin{align*}
	\E \int_0^\eps t \nu_{n}(dt)&=\frac{1}{z_n} \sum_{k:z_k \le z_n \eps} z_k \E X_k
	\le \frac{C}{z_n} \sum_{k:z_k \le z_n \eps} z_k q_k^1\\
	&\le \frac{C}{z_n} \sum_{k:z_k \le z_n \eps}z_k
	\left(1-\left(1-\frac{z_k-z_{k-1}}{z_k}\right)^c\right)\\
	&\le \frac{C}{z_n} \sum_{k:z_k \le z_n \eps} 2^{\lceil c \rceil}(z_k-z_{k-1})
	\le C2^{\lceil c \rceil}\eps.
	\end{align*}
	Therefore,
	\begin{equation}\label{eq:rem4.5}
	\lim_{\eps \to 0} \limsup_{n \to \infty} \E \int_0^\eps t \nu_{n}(dt)=0.
	\end{equation}
	Thus, invoking Lemma~\ref{lem:gensum} we obtain
	\begin{align*}
	&\int_0^1 t \nu_{n}(dt)=\frac{1}{z_n} \sum_{k=1}^n z_k X_k \vd  D_c \quad \text{as }\;n \to \infty. \qedhere
	\end{align*}
\end{proof}

\begin{remark}\label{rem.ber} 
	Recall that $X$ is a geometric random variable with parameter $p
	\in(0,1)$ if $\Prob{X=m} =(1-p)^m p$ for $m \ge 0$; we then write $X
	\sim {\rm Geom}(p)$. For $(z_k)_{k\in\N}$ as in Theorem~\ref{ex},
	clearly $X_k \sim {\rm Geom}(q_k^0)$ satisfies the conditions therein. We
	can also take the random variables $X_k \sim {\rm Ber}(q_k^0)$ with
	$q_k^0$ as in Theorem~\ref{ex}, i.e.\ $X_k$ is a $\{0,1\}$-valued
	random variable with $\Prob{X_k=0}=q_k^0$. In this case, a similar
	proof shows that
	\begin{displaymath}
	\nu_{n}=\sum_{k=1}^\infty X_k \delta_{z_k/z_n}
	\vd \eta_c \quad \text{as }\; n \to \infty.
	\end{displaymath}
	Since $\E X_k=q_k^1$, arguing like in Theorem~\ref{ex}, one can establish \eqref{Dickman} in this case as well.
\end{remark}

\begin{remark}\label{rem.1.4}
	Even though under Condition~\ref{cond.1d} the sequence $\nu_{n}$
	converges vaguely in distribution to a simple process, it is not necessarily true
	that the $X_k$'s are $\{0,1\}$-valued almost surely for all
	sufficiently large $n$. Consider the sequence $\nu_{n}$ as in
	Theorem~\ref{ex} with $c=1$ and $z_k$ defined sequentially by letting $z_0=z_1=1$
	and $z_n/z_{n-1}=\sqrt{n}/(\sqrt{n}-1)$ for $n \ge 2$. Since
	\begin{displaymath}
	z_n = \frac{\sqrt{n}}{\sqrt{n}-1}z_{n-1}\ge \frac{n}{n-1}z_{n-1}\ge \cdots
	\ge n z_1=n,
	\end{displaymath}
	Theorem~\ref{ex} yields that $\nu_{n} \vd \eta$ as $n \to
	\infty$. Furthermore, 
	\begin{displaymath}
	\sum_{k=1}^\infty \Prob{X_k\geq 2}
	=\sum_{k=1}^\infty (1-q_k^0-q_k^1)
	=\sum_{k=2}^\infty (1-z_{k-1}/z_k)^2 \ge \sum_{k=2}^\infty k^{-1},
	\end{displaymath}
	which diverges. By the Borel-Cantelli lemma, $X_k$ is strictly
	greater than 1 for infinitely many $k$. However, after rescaling,
	the number of points with multiplicities more than 1 in any bounded
	interval $[a,b]\subset(0,\infty)$ converges to zero.
	
	The processes in Theorem~\ref{ex} do not necessarily satisfy
	\eqref{cond.sum}, since only $q_k^0$ and $q_k^1$ are specified there
	and one can allocate the rest of the probability on a large number
	to make $\E X_k$ sufficiently large so that \eqref{cond.sum} does
	not hold. Hence, an additional condition like $\E
	X_k=\mathcal{O}(q_k^1)$ is essential. Note that, for $X_k \sim {\rm
		Geom}(q_k^0)$, we have $q_k^1=q_k^0(1-q_k^0)$ and
	\begin{displaymath}
	\E X_k=(1-q_k^0)/q_k^0=(1/q_k^0)^2 q_k^1=\mathcal{O}(q_k^1),
	\end{displaymath}
	since $q_k^0 \to 1$ as $k \to \infty$.
\end{remark}

Next, we describe a sequence of point processes arising in
probabilistic number theory which satisfies Condition~\ref{cond.1d}, and hence, converges to the scale-invariant Poisson
process $\eta$ by Theorem~\ref{thm2} and the sums of points in $(0,1)$ converge to the standard
Dickman distribution. For an enumeration $(p_k)_{k\in\N}$ of the
prime numbers in increasing order, let $\Omega_{n}$ denote the set of
positive integers having all its prime factors less than or equal
to the $n^{th}$ prime $p_n$. Let $M_n$ be a random variable
distributed according to the probability mass function $\Theta_n$ with
$\Theta_n(m)$ being proportional to the inverse of $m$ for $m \in
\Omega_n$.  Then one can show that (see e.g.\ \cite{Pi16})
\begin{equation}
\label{eq:1}
\frac{\log M_{n}}{\log p_n}
=_d \frac{1}{\log p_n} \sum_{k=1}^n X_k \log p_k,
\end{equation}
where $X_1,\ldots,X_n$ are independent with $X_k \sim {\rm
	Geom}(1-1/p_k)$ for $1 \le k \le n$.  The distributional convergence
of the right-hand side of \eqref{eq:1} to the standard Dickman distribution was proved in \cite{Pi16} with optimal convergence rates provided in \cite{bh17} using Stein's
method. We prove that this convergence is a consequence of the
underlying sequence of point processes converging to $\eta$.

\begin{theorem}\label{thm.Dickman}
	Let $(\nu_{n})_{n\in\N}$ be a sequence of point processes defined
	at \eqref{eq:nun} with $z_k=\log p_k$ and $X_k \sim {\rm
		Geom}(1-1/p_k)$ for $k\in\N$. Then $\nu_{n} \vd \eta$ as $n \to
	\infty$ and
	\begin{displaymath}
	\frac{1}{\log p_n} \sum_{k=1}^n X_k \log p_k \vd D
	\quad \text{as }\; n \to \infty.
	\end{displaymath}
\end{theorem}

\begin{proof}
	For the first part, by Theorem~\ref{thm2}, we only need to check
	Condition~\ref{cond.1d}. Since $q_k^0=(1-1/p_k)$, for $0<a<b<\infty$, by Merten's
	formula (see e.g.\ \cite[Prop.~1.51]{tao06}), 
	\begin{displaymath}
	\prod_{k: az_n < z_k \le bz_n} q_k^0= \prod_{k: p_n^a < p_k
		\le p_n^b} \left(1-\frac{1}{p_k}\right) \to \frac{a}{b} 
	\quad \text{as }\; n \to \infty.
	\end{displaymath}
	Hence, Condition~\ref{cond.1d}(i) is satisfied. For
	Condition~\ref{cond.1d}(ii), since $q_k^1=p_k^{-1}(1-p_k^{-1})$, 
	Merten's formula yields that
	\begin{displaymath}
	\sum_{k: az_n < z_k \le bz_n}q_k^1/q_k^0
	=\sum_{k: p_n^a < p_k
		\le p_n^b}\frac{1}{p_k}\to \log \frac{b}{a} \quad \text{as }\; n \to \infty.
	\end{displaymath}
	Theorem~\ref{thm2} now yields the first part of the result.
	
	For the second part, by Lemma~\ref{lem:gensum}, it suffices to check
	\eqref{cond.sum}.  Since 
	\begin{displaymath}
	\sum_{p_k\le n} p_k^{-1}\log p_k=\log n + \mathcal{O}(1)
	\end{displaymath}
	(see \cite[Prop.~1.51]{tao06}), it follows that for $\eps >0$,
	\begin{displaymath}
	\E \int_0^\eps t \nu_{n}(dt)=\frac{1}{\log p_n}\E \sum_{k=1}^\infty X_k
	\log p_k \mathds{1}_{\{1< p_k \le p_n^\eps\} } \le \frac{2}{\log p_n}
	[\log p_n^\eps + \mathcal{O}(1)],
	\end{displaymath}
	which converges to $\eps$ as $n \to \infty$. Thus, $(\nu_n)_{n \ge
		1}$ satisfies \eqref{cond.sum}, proving the result.
\end{proof}

\begin{remark}\label{exp1}
	Let $X_k \sim {\rm Ber}(1/(1+p_k))$, where $p_k$ is the $k^{th}$
	prime number and consider $(\nu_{n})_{n\in\N}$ defined in 
	Theorem~\ref{thm.Dickman}. One can argue as in the proof
	of Theorem~\ref{thm.Dickman} to show that $\nu_{n} \vd \eta$ as $n \to
	\infty$ and 
	\begin{displaymath}
	\frac{1}{\log p_k} \sum_{k=1}^n X_k \log p_k \vd D\quad \text{as
	}\; n\to\infty.
	\end{displaymath}
	As mentioned above, if the $X_k$'s are distributed as geometric random variables given in Theorem~\ref{thm.Dickman}, the
	induced distribution on $M_n=\prod_{k=1}^n p_k^{X_k}$ is the
	reciprocal distribution on the set $\Omega_n$ of positive integers
	with all prime factors less than or equal to $p_n$. If $X_k \sim {\rm
		Ber}(1/(1+p_k))$, the induced distribution on $M_n$ turns out to
	be the reciprocal distribution on the set of square-free positive
	integers with all its prime factors less than or equal to $p_n$.
\end{remark}

Next, we provide a few more examples that arise as special cases of the
class of point processes considered in Theorem~\ref{thm2} and in
Remark~\ref{rem.ber}.

\begin{example}\label{exp2}
	Let $X_k \sim {\rm Ber}(1/k)$, $k\geq1$, be independent and
	$\nu_n=\sum_{k=1}^\infty X_k \delta_{k/n}$. In this case, one can
	easily check that Condition~\ref{cond.1d} and~\eqref{cond.sum} are
	satisfied. Hence, $\nu_n \vd \eta$ and $n^{-1}\sum_{k=1}^n k X_k \vd
	D$ as $n \to \infty$. This is a well-known example arising in the
	context of counting sums of `records' in a random permutation. For a
	uniformly random permutation $\sigma$ of $\{1,\dots,n\}$, let $S_n$
	be the sum of records, which are positions $k$ such that
	$\sigma(k)>\max_{i \in [k-1]}\sigma(i)$. One can check that $S_n$ is
	indeed distributed as $\sum_{k=1}^n k X_k$.
\end{example}

\begin{example}\label{exp3}
	Let $\nu_{n}$ be as in \eqref{eq:nun} with $z_k=\log k$ and
	independent $X_k \sim {\rm Geom}(1-1/(k\log k))$, $k\in\N$. In
	this case, it is straightforward to check that the conclusions of
	Theorem~\ref{ex} hold. Heuristically, this is equivalent to Theorem~\ref{thm.Dickman}, since, by the prime number theorem, one has that the $k^{th}$ prime number $p_k$ is asymptotically of the order $k \log k$.
\end{example}

\begin{example}\label{exp4}
	Theorem~\ref{thm2} and Lemma~\ref{lem:gensum} apply if $X_k$'s are
	independent Poisson random variables with mean $1/p_k$ and $\nu_{n}$
	is given by \eqref{eq:nun} with $z_k=\log p_k$, $k\in\N$.
\end{example}

\section{Convergence of uplifted point processes}
\label{sec.hd}

In this section, we consider convergence of certain general point
processes to scale-invariant Poisson processes in dimension $d$.
These point processes are obtained by first taking a point process on
$(0,\infty)$ and transforming (uplifting) its points to $\R^d$ by
multiplying them with random vectors taking values in $S=\R^d \setminus\{0\}$. We
start with a point process $\xi=\sum_{k=1}^\infty X_k \delta_{Z_k}$ with finite intensity on the positive half-line. Let $V$ be a random vector in $S$ with
i.i.d.\ copies $(V_k)_{k \in \N}$ which are independent of $\xi$. Define the uplifted process $\xi^V$ as
\begin{equation}\label{eq:iidup}
\xi^V= \sum_{k=1}^\infty X_k \delta_{V_k Z_k}.
\end{equation} 
We need to impose some conditions on $\xi$ and $V$ to ensure that $\xi^V$ is locally finite on $S$. To this end, throughout this section, we assume for any uplifted process $\xi^V$ that $\xi$ and $V$ satisfy
\begin{equation}\label{eq:condwd}
\E \sum_{k=1}^\infty X_k \mathds{1}_{Z_k \|V_k\| \in [a,b]} <\infty \quad \text{for all }\; 0<a<b<\infty,
\end{equation}
where $\|\cdot\|$ denotes the Euclidean norm.
Since $\xi$ has a finite intensity, this condition is always satisfied if $V$ is bounded away from $0$ and $\infty$. In Lemma~\ref{lem.char}, we show that
any scale-invariant Poisson process in $S$ has the same distribution
as the uplifted process $\eta_c^U$ for some $c>0$ and a unit random
vector $U$ in $\R^d$. Thus, our uplifting scheme is a natural
choice to recover all scale-invariant point processes in $S$.

It is well known that, if $\xi_n \vd \xi$ as $n \to \infty$, then (see e.g.\ \cite[Theorem~4.11]{kalle17})
\begin{equation}
\label{bddf}
\E\; e^{-\xi_n f} \to \E\; e^{-\xi f} \quad \text{as }\; n \to \infty
\end{equation}
for any $f \in \widehat C_{S}$. In order to handle uplifting
transformations by a possibly unbounded random vector $V$, we need to
consider test functions $f$ with unbounded support. The
following result extends \eqref{bddf} to more general functions.

\begin{lemma}\label{lem1.1'}
	Let $(\xi_n)_{n\in\N}$ and $\xi$ be point processes on a locally
	compact separable metric space $\Omega$ with $\xi$ having a finite intensity, such that $\xi_n \vd \xi$
	as $n \to \infty$. Let $h$ be a non-negative continuous function on
	$\Omega$ such that for any $\eps>0$, there exists a relatively
	compact set $K_\eps$ with
	\begin{equation}\label{cond.sum'} 
	\limsup_{n \to \infty}\E \int_{K_\eps^c} h(x) \xi_n(dx)
	\le \eps.
	\end{equation} 
	Then
	\begin{displaymath}
	\E\; e^{-\xi_n h} \to \E\; e^{-\xi h} \quad \text{as }\; n \to \infty.
	\end{displaymath}
\end{lemma}

For a proof, see the Appendix. For $f \in \widehat C_S$, define the function
$h_f:\N_0\times(0,\infty) \to \R$ as
\begin{equation}
\label{eq:4}
h_f(x,y)=-\log \E \; e^{-x f(V y)}.
\end{equation}
Note that by Jensen's inequality, one has
\begin{equation}\label{eq:Jen}
h_f(x,y) \le x \E\; f(V y).
\end{equation}
Define the map $M:\sN((0,\infty)) \to
\sN(\N_0\times(0,\infty))$ at $\xi=\sum_{k=1}^\infty a_k \delta_{z_k}$
as
\begin{equation}
\label{eq:T}
M(\xi)= \sum_{k=1}^\infty \delta_{(a_k,z_k)}.
\end{equation}
This map turns a counting measure with possibly multiple points into a
simple counting measure in the product space $\N_0 \times
(0,\infty)$. 

\begin{theorem}
	\label{thm.unbdd}
	Assume that a sequence of point processes $\xi_n=\sum_{k=1}^\infty X_k \delta_{Z_k^n}$, $n \in \N$
	converges vaguely in distribution to a simple point process $\xi$ with finite intensity in
	$\sN((0,\infty))$ as $n \to \infty$. Moreover, let $V$ be a
	random vector in $S$ with i.i.d.\ copies $(V_k)_{k \in \N}$ such that for every $f \in \widehat C_S$ and $\eps>0$, there exists a compact set $K_{f,\eps} \subseteq \N_0 \times (0,\infty)$ such that
	\begin{equation}\label{eq:tra}
	\limsup_{n \to \infty}\E \sum_{(X_k,Z_k^n) \in K_{f,\eps}^c} X_k f(V_k Z_k^n) \le \eps.
	\end{equation} 
	Then $\xi_n^{V} \vd \xi^{V}$ as $n \to \infty$.
\end{theorem}

\begin{proof}
	Fix $f\in \widehat C_S$. Then
	\begin{align*}
	\E\; e^{-\xi_ { n }^{V} f }&=\E \left[\prod_{k=1}^\infty \E\left[\exp\left\{-X_k
	f(V_k Z_k^n)\right\}\Big|\xi_n \right]\right]\\
	&=\E\; \exp\Big\{-\sum_{k=1}^\infty
	h_f (X_k,Z_k^n)\Big\}=\E\; e^{-\widetilde \xi_n h_f },
	\end{align*}
	where $\widetilde \xi_n=M(\xi_n)$ and $h_f$ is given by \eqref{eq:4}. Since $\xi_n \vd \xi$ as $n \to
	\infty$ with $\xi$ being simple, Lemma~\ref{lem:contT} and the
	continuous mapping theorem yield that $\widetilde \xi_n \vd
	\widetilde \xi=M(\xi)$. Clearly, $h_f$ is continuous as $f$ is
	such. Also note that by \eqref{eq:Jen} and \eqref{eq:tra}, we have that $h_f$ satisfies \eqref{cond.sum'} with respect to the processes $(\widetilde \xi_n)_{n \in \N}$. By Lemma~\ref{lem1.1'},
	\begin{displaymath}
	\E\; e^{-\xi_{ n }^V f}=\E\; e^{-\widetilde \xi_n h_f} \to \E\;
	e^{-\widetilde \xi h_f} \quad \text{as }\; n \to \infty.
	\end{displaymath} 
	Finally, noticing that
	\begin{displaymath}
	\E\; e^{-\xi^{V} f}=\E\big[\E(e^{-\xi^{V} f}| \xi)\big] = \E\; e^{-\widetilde \xi h_f},
	\end{displaymath}
	we obtain
	\begin{displaymath}
	\E\; e^{-\xi_ { n }^V f } \to \E\; e^{-\xi^{V} f}
	\quad \text{as }\; n\to\infty
	\end{displaymath} 
	for all $f \in \widehat C_S$, which proves that $\xi_n^V \vd
	\xi^{V}$ as $n \to \infty$.
\end{proof}

The condition \eqref{eq:tra} in Theorem~\ref{thm.unbdd} that $(V_n)_{n \in \N}$ and $(\xi_n)_{n \in \N}$ are required to satisfy can be hard to check in general. In
some special cases, one can find some easily verifiable conditions on
$(\xi_n)_{n\in\N}$ and $V$ so that \eqref{eq:tra} is
satisfied. Throughout, $\|\cdot\|_\infty$ denotes the supremum
norm on $\widehat C_S$.

\begin{lemma}\label{lem2.8}
	Let $(\xi_n)_{n\in\N}$ be a sequence of simple point processes in
	$(0,\infty)$. Let $V$ be such that for some $\alpha>0$,
	\begin{equation}\label{cond.u}
	\limsup_{t \to \infty}t\Prob{\|V\|\ge t}<\infty \quad \text{and}\quad 
	\limsup_{t \to \infty}t^\alpha \Prob{\|V\|\le 1/t}<\infty.
	\end{equation} 
	Moreover, assume that 
	\begin{equation}\label{cond.nu}
	\lim_{r \to 0} \limsup_{n \to \infty} \E \int_0^r t \xi_n(dt)=0 \quad
	\text{and}\quad \lim_{r \to \infty} \limsup_{n \to \infty}
	\E \int_r^\infty t^{-\alpha} \xi_n(dt)=0.
	\end{equation}
	Then the processes $(\xi_n)_{n \in \N}$ and i.i.d.\ copies $(V_n)_{n \in \N}$ of $V$ satisfy \eqref{eq:tra}.
\end{lemma}

\begin{proof}
	Since $\xi_n$ is simple, for $f \in \widehat C_{S}$, it suffices to check that $h(y)=\E f(Vy)$ satisfies
	\begin{equation}
	\label{cond.sum''} 
	\begin{cases}
	\lim_{r \to 0} \limsup_{n \to \infty} \E \int_0^r h(y) \xi_n(dy)
	=0,\\
	\lim_{r \to \infty} \limsup_{n \to \infty}
	\E \int_r^\infty h(y) \xi_n(dy)=0.
	\end{cases}
	\end{equation} 
	Since $f$ is compactly supported, there exist $0<a<b<\infty$ such that $f(z)=0$
	for $\|z\|<a$ or $\|z\|>b$. Thus, using \eqref{cond.u} in the last
	step, we have
	\begin{align*}
	\limsup_{y \searrow 0}\frac{h(y)}{y} &= \limsup_{y \searrow 0}
	\frac{\E\; f(Vy)}{y} = \limsup_{y \searrow 0}
	\frac{\E\; \big[f(Vy)\mathds{1}_{\{\|V\|\ge a/y\}}]}{y}\\
	&\le\|f\|_\infty \limsup_{y \searrow 0} y^{-1}\Prob{\|V\|
		\ge a/y}<\infty.
	\end{align*}
	Arguing similarly and using \eqref{cond.u},
	\begin{align*}
	\limsup_{y \to \infty}y^\alpha h(y)
	&= \limsup_{y\to \infty} y^\alpha \E\; f(Vy)\\
	&=\limsup_{y \to \infty} y^{\alpha}
	\E\; \big[f(Vy)\mathds{1}_{\{\|Vy\| \le b\}}\big]
	\\& \le\|f\|_\infty \limsup_{y \to \infty} y^{\alpha}
	\Prob{\|V\| \le b/y}<\infty.
	\end{align*}
	Thus, $\limsup_{y \searrow 0} h(y)/y <\infty$ and
	$h(y)=\mathcal{O}(y^{-\alpha})$ as $y \to \infty$.
	Together with \eqref{cond.nu}, this implies that $h$
	satisfies \eqref{cond.sum''}.
\end{proof}

\begin{corollary}\label{cor:etaconv}
	Let $(\xi_n)_{n\in\N}$ be a sequence of simple point processes
	converging vaguely in distribution to $\eta_c$ as $n \to
	\infty$. Assume that a random vector $V$ in $S$ and
	$(\xi_n)_{n\in\N}$ satisfy \eqref{cond.u} and \eqref{cond.nu},
	respectively, for some $\alpha>0$. Then $\xi_n^V \vd \eta_c^V$ as $n
	\to \infty$.
\end{corollary}

\begin{remark}
	Fix $\alpha>0$. For a sequence of point processes $(\nu_n)_{n\in\N}$
	as in Theorem~\ref{ex} with $\E X_k=\mathcal{O}(q_k^1) \le C q_k^1$
	for some $C>0$, by \eqref{eq:rem4.5} in the proof of Theorem~\ref{ex}, the first condition in \eqref{cond.nu} is satisfied. Letting
	$N=\inf\{k:z_k>z_nr\}$ for $r>0$ yields that
	\begin{align*}
	\E \int_r^\infty t^{-\alpha} \nu_n(dt)
	=z_n^{\alpha} \sum_{k:z_k > z_n r} & z_k^{-\alpha}
	\E X_k \le Cz_n^\alpha (\sup_k q_k^0)
	\sum_{k=N}^\infty \frac{z_k^c-z_{k-1}^c}{z_k^{c+\alpha}} \\
	\le &Cz_n^\alpha \left[\frac{z_N^c-z_{N-1}^c}{z_N^{c+\alpha}}
	+ \int_{(z_n r)^c}^\infty \frac{1}{x^{(c+\alpha)/c}} dx \right].
	\end{align*}
	Since the right-hand side converges to $C (c/\alpha)r^{-\alpha}$ as $n
	\to \infty$,
	\begin{displaymath}
	\lim_{r \to \infty} \limsup_{n \to \infty} \E \int_r^\infty t^{-\alpha} \nu_n(dt)=0.
	\end{displaymath}
	Hence, these point processes satisfy \eqref{cond.nu}.
\end{remark}

\begin{example}
	Consider the sequence of point processes $(\nu_n)_{n\in\N}$ given
	by \eqref{eq:nun} with $z_k=\log p_k$ and $X_k \sim {\rm
		Geom}(1-1/p_k)$. Since $p_k>k \log k$ (see e.g.\ \cite{tao06}) and $\log p_k< 2\log k$ for $k \ge 6$, (see e.g.\ \cite[Lem.~1]{Du99}), we have that $N_n=\inf \{k: p_k>p_n^r\} > n^{r/2}$ for $n$ large enough. Hence,
	\begin{multline*}
	\sum_{k:p_k>p_n^r} \frac{1}{p_k (\log p_k)^\alpha}
	\le \sum_{k=N_n}^\infty \frac{1}{k \log k (\log k)^\alpha}\\
	\le \int_{N_n-1}^\infty \frac{1}{x(\log x)^{1+\alpha}} dx=\frac{(\log (N_n-1))^{-\alpha}}{\alpha} \le \frac{2^\alpha (\log n)^{-\alpha}}{\alpha r^{\alpha}}.
	\end{multline*}
	Therefore,
	\begin{align*}
	\limsup_{n \to \infty}\E \int_r^\infty t^{-\alpha} \nu_n(dt)
	&=\limsup_{n \to \infty} (\log p_n)^\alpha\sum_{k:p_k>p_n^r} \frac{2}{p_k (\log p_k)^\alpha}\\
	&\le\limsup_{n \to \infty}(2 \log n)^\alpha \frac{2^{1+\alpha}(\log n)^{-\alpha}}{\alpha r^{\alpha}} =\frac{2^{1+2\alpha}}{\alpha r^{\alpha}},
	\end{align*}
	which yields
	\[
	\lim_{r \to \infty}\limsup_{n \to \infty}\E \int_r^\infty t^{-\alpha} \nu_n(dt)=0.
	\]
	The other condition in \eqref{cond.nu} is easy to check using Merten's formula. Hence, as $\nu_n \vd \eta$ as $n \to \infty$ by Theorem~\ref{thm.Dickman}, for $V$ satisfying
	\eqref{cond.u}, Corollary~\ref{cor:etaconv} yields that $\nu_n^V \vd
	\eta^V$ as $n \to \infty$.
\end{example}

We now return to our basic example of point processes given by
\eqref{eq:nun}. For a point process on $(0,\infty)$ with support points in a deterministic set, we can generalize the notion of uplifting. For $(\nu_{n})_{n\in\N}$ given by \eqref{eq:nun}, consider its uplifting by independent vectors $\bv=(V_k)_{k \in \N}$ in $S$ which are possibly non-identically distributed, allowing for possible dependence within the pairs $(V_k,X_k)$ for any $k \in \N$. Assume that the conditional distribution of $V_k$ given $X_k$ is a function $V(X_k)$ that does not depend on $k$, i.e.,
\begin{equation}\label{eq:vcond}
V(x)=_d (V_k|X_k=x), \quad k\in\N.
\end{equation}
For instance, this is the case if the random vectors $(V_k)_{k \in \N}$ are i.i.d.\ and independent of the random variables $(X_k)_{k \in \N}$.  We also assume that the random vectors $(V_k)_{k \in \N}$ are uniformly bounded away from $0$ and $\infty$ and define the \emph{uplifted} process $\nu_n^\bv$ as
\begin{displaymath}
\nu_n^\bv= \sum_{k=1}^\infty X_k \delta_{V_k z_k/z_n}.
\end{displaymath}
Finally, we assume that the random
variables $(X_k)_{k\in\N}$ are $\{0,1\}$-valued with high probability, i.e.,
\begin{equation}
\label{cond1}
\prod_{k=1}^\infty (q_k^0+q_k^1)>0.
\end{equation}

\begin{theorem}\label{thm4}
	For
	$(\nu_{n})_{n\in\N}$ given by \eqref{eq:nun}, assume that the
	$X_k$'s satisfy \eqref{cond1} and $\nu_{n} \vd \eta_c$ for some
	$c>0$. Let $\bv=(V_k)_{k\in\N}$ be a sequence of random vectors in $S$
	satisfying \eqref{eq:vcond} with $\eps \le \|V(x)\|\leq r$ almost
	surely for all $x \in \N$ for some $0<\eps<r<\infty$. Then $\nu_{n}^{\bv} \vd \eta_c^{V(1)}$ as $n \to
	\infty$, where $\eta_c^{V(1)}$ is defined as in \eqref{eq:iidup}.
\end{theorem}

\begin{proof}
	Let $\widetilde X_k=\mathds{1}_{\{X_k>0\}}$. Let $\left(m(n)\right)_{n
		\in \N}$ be such that $m(n) \to \infty$ and $z_{m(n)}=o(z_n)$ as $n \to \infty$. Denote
	\begin{displaymath}
	E_n=\{X_k= \widetilde X_k \text{ for all } k
	\ge m(n)\}, \;\; n \in \N.
	\end{displaymath}
	By Kolmogorov's zero-one law and \eqref{cond1}, 
	\begin{displaymath}
	\lim_{n \to \infty}\P(E_n)=1.
	\end{displaymath} 
	Fix $f \in \widehat C_S$. Then, recalling that $V(1)=_d (V_k|X_k=1)$, we have
	\begin{align}\label{exp.split}
	&\E \left[e^{-\nu_ {n}^{\bv} f } | E_n\right]
	=\E \left[\exp\Big\{-\sum_{k=1}^\infty
	X_k f(V_k z_k/z_n)\Big\}\Big|E_n\right]\nonumber\\
	&=\E \left[\exp\Big\{-\sum_{k=1}^{m(n)-1}
	X_k f(V_k z_k/z_n)\Big\}\right]
	\E \left[\prod_{k=m(n),X_k=1}^\infty \E\; e^{-
		f(V(1) z_k/z_n)}\Big | E_n\right].
	\end{align}
	Since $z_{m(n)}=o(z_n)$, the process $\sum_{k=1}^{m(n)-1} X_k
	\delta_{z_k/z_n}$ converges vaguely in distribution to the zero process in $\mathcal{M}((0,\infty))$ as $n \to \infty$.
	Combined with our assumption that $\eps \le \|V(x)\| \le r$ almost surely for all $x \in \N$, this implies that the first factor on the right-hand side of \eqref{exp.split}
	converges to 1 as $n \to \infty$. For the second factor in
	\eqref{exp.split}, we have
	\begin{displaymath}
	\E \left[\prod_{k=m(n),X_k=1}^\infty \E\; e^{-
		f(V(1) z_k/z_n)}\Big | E_n\right]
	=\E\; \exp\Big\{-\sum_{k=m(n)}^\infty Y_k \widetilde h(z_k/z_n)\Big\},
	\end{displaymath}
	where $Y_k \sim {\rm Ber}(q_k^1/(q_k^0+q_k^1))$, $k \ge m(n)$, has
	the same distribution as $X_k$ conditional on $E_n$, and
	\begin{displaymath}
	\widetilde h(t)=-\log \E\; e^{-f(V(1)t)}.
	\end{displaymath}
	Consider the point process
	$\widetilde \nu_{n}=\sum_{k=1}^\infty Y_k \delta_{z_k/z_n}$. Using
	\eqref{cond1} for the first equality, we have that for any
	$0<a<b<\infty$,
	\begin{equation*}
	\lim_{ n \to \infty }\prod_{k: az_n < z_k \le bz_n}
	\frac{q_k^0}{q_k^0+q_k^1}
	=\lim_{ n \to \infty }\prod_{k: az_n < z_k \le bz_n} q_k^0=\left(\frac{a}{b}\right)^c,
	\end{equation*}
	where in the last equality we have used our assumption that $\nu_{n}
	\vd \eta_c$ and Theorem~\ref{thm2}. Hence, $(\widetilde \nu_{n})_{n \in \N}$ satisfies Condition~\ref{cond.1d}(i). That $(\widetilde \nu_{n})_{n\in\N}$ satisfies
	Condition~\ref{cond.1d}(ii) follows trivially by noticing that
	$(\nu_{n})_{n\in\N}$ satisfies
	Condition~\ref{cond.1d}(ii). Thus, $\widetilde \nu_{n}$ converges vaguely in distribution
	to $\eta_c$ as $n \to \infty$ by Theorem~\ref{thm2}. Again, we can
	ignore the first $m(n)-1$ terms of the sum $\widetilde \nu_{n}$ as
	it converges to a zero process, whence
	\begin{displaymath}
	\sum_{k=m(n)}^\infty Y_k \delta_{z_k/z_n} \vd \eta_c
	\quad \text{as }\; n \to \infty.
	\end{displaymath}
	By our assumption that $V(1)$ is bounded away from $0$ and $\infty$ and that $f$ is compactly supported, it follows that the function $\widetilde h$ has a relatively compact support in $(0,\infty)$. Clearly,
	$\widetilde h$ is continuous and bounded. Hence by \eqref{bddf},
	\begin{displaymath}
	\E\; \exp\Big\{-\sum_{k=m(n)}^\infty
	Y_k \widetilde h(z_k/z_n)\Big\} \to \E\; e^{-\eta_c \widetilde
		h} 
	\quad\text{as }\;n \to \infty.
	\end{displaymath} 
	By \eqref{exp.split},
	\begin{displaymath}
	\E \left[e^{-\nu_ {n}^\bv f } | E_n\right]\to \E\; e^{-\eta_c
		\widetilde h} \quad \text{as }\; n \to \infty.
	\end{displaymath}
	Finally, noticing that
	\begin{displaymath}
	\E\; e^{-\eta_c \widetilde h}=\E\Big[\E(e^{-\eta_c^{V(1)} f}| \eta_c)\Big] = \E\; e^{-\eta_c^{V(1)} f},
	\end{displaymath}
	and that $\P(E_n) \to 1$ as $n\to\infty$, we have
	\begin{displaymath}
	\E\; e^{-\nu_{n}^\bv f}
	=\E \left[e^{-\nu_{n}^\bv f} |E_n\right]\P(E_n)
	+ \E \left[e^{-\nu_{n}^\bv f} |E_n^c\right]\P(E_n^c) 
	\to \E\; e^{-\eta_c^{V(1)} f}
	\end{displaymath} 
	as $n \to \infty$ for any $f \in \widehat C_S$, which yields that
	$\nu_{n}^\bv \vd \eta_c^{V(1)}$ as $n \to \infty$.
\end{proof}

\begin{example}
	For $z_k=\log p_k$ and $X_k\sim {\rm Geom}(1-1/p_k)$ or ${\rm Ber}(1/(p_k+1))$, one can easily see that
	the conditions of Theorem~\ref{thm4} are satisfied by $(\nu_{n})_{n
		\in \N}$, and hence, for $\bv$ as in Theorem~\ref{thm4}, the
	conclusion of the result holds.
\end{example}

Note, if $V_k$ is independent of $X_k$ for all $k \in
\N$, then they are necessarily i.i.d.\ by \eqref{eq:vcond}. Now we consider an example when $(X_k)_{k\in\N}$ and
$\bv$ are dependent.

\begin{example}\label{ex.mult}
	Let $d \ge 2$ and $m \in \N$ be positive integers. Let $X_k \sim
	{\rm Geom}(1-1/p_k)$ be independent and $V_k=(mX_k)^{-1}(X_k^1,
	\dots, X_k^d)\mathds{1}_{\{X_k>0\}}$ for $k\in\N$, where $(X_k^1,
	\dots, X_k^d)$ is multinomially distributed with the number of
	experiments $mX_k$ and the probabilities of outcomes
	$q_1,\dots,q_{d}$ with $\sum_{i=1}^{d}q_i =1$. Let
	\begin{displaymath}
	\nu_n=\sum_{k=1}^\infty X_k\delta_{\log p_k/\log p_n}
	\quad\text{and}\quad 
	\nu_{n}^\bv=\sum_{k=1}^\infty X_k \delta_{V_k \log p_k/\log p_n},
	\end{displaymath}
	where $(p_k)_{k\in\N}$ is an enumeration of the primes. Clearly,
	the random variables $(X_k)_{k\in\N}$ satisfy \eqref{cond1}. For
	each $k$, the random vector $V_k$ and hence $V(x)$ is almost surely bounded away from
	$0$ and $\infty$ when $X_k=x>0$. Since by Theorem~\ref{thm.Dickman} we have that $\nu_n$ converges vaguely in
	distribution to $\eta$ as $n \to \infty$, Theorem~\ref{thm4} yields
	that $\nu_n^\bv \vd \eta^{V(1)}$ as $n \to \infty$, where $m V(1)$ is
	distributed as a multinomial random variable with $m$ experiments and probabilities
	of outcomes $q_1,\dots,q_{d}$. 
\end{example}

\section{Scale-invariant Poisson processes in higher dimensions and
	multivariate Dickman distributions}
\label{sec:scale-invar-high}

In this section, we study and classify scale-invariant Poisson
processes in higher dimensions and extend the generalized Dickman
distributions in one dimension to its multivariate counterpart. For a simple
point process $\xi$ in $(0,\infty)$ and a random vector $V$ taking values in $S=\R^d \setminus\{0\}$ bounded away from $0$ and $\infty$ with i.i.d.\ copies
$(V_k)_{k\in\N}$, recall that the uplifted point
process $\xi^V$ is given by
\begin{equation*}
\xi^V=_d \sum_{k=1}^\infty \delta_{V_k Z_k},
\end{equation*}
where $(Z_k)_{k\in\N}$ is an enumeration of the points in $\xi$.

\begin{lemma}\label{lem.char}
	Any scale-invariant Poisson process in $S$ has the same
	distribution as $\eta_c^U$ for some $c>0$ and unit random vector $U$
	in $\R^d$. Moreover, for any random vector
	$V$ in $S$ with $\eta_c$ and $(V_k)_{k\in\N}$ satisfying \eqref{eq:condwd}, the uplifted point process $\eta_c^V$ has the same distribution as $\eta_c^U$  with $U= V/\|V\|$.
\end{lemma}

\begin{proof}
	Let $\nu$ be a scale-invariant Poisson process in $S$. Hence
	$\nu(tB)=_d\nu(B)$ for every Borel set $B \in \fS$ and
	$t>0$. Represent each point $x \in S$ as a pair $(u,r) \in
	\Sphere\times (0,\infty)$, where $u=x/\|x\|$ and $r=\|x\|$. For a
	measurable subset $A \subseteq \Sphere$ and $0<a<b<\infty$, by scale
	invariance one has
	\begin{equation}\label{eq:numean}
	\E \;\nu(A\times[a,b])=\E\;\nu(A\times [a/b,1]).
	\end{equation}
	For $p \in(0,1)$ and $A \subseteq \Sphere$, define $\gamma_{\nu}(p,A)=\E \;\nu(A \times[p,1])$
	and $\gamma_{\nu}(1,A)=0$. For every fixed $A\subseteq \Sphere$, notice that
	$\gamma_\nu$ satisfies
	\begin{displaymath}
	\gamma_{\nu}(p,A)+\gamma_{\nu}(q,A)=\gamma_{\nu}(pq,A),
	\quad p,q \in(0,1).
	\end{displaymath}
	By monotonicity, $\gamma_{\nu}(p,A)=-\gamma_\nu(A)\log p$ for $p \in
	(0,1]$, where $\gamma_\nu$ is a locally finite measure on $\Sphere$
	not depending on $p$. By \eqref{eq:numean}, 
	\begin{displaymath}
	\E\;\nu(A\times[a,b])=\gamma_\nu(A)\log (b/a). 
	\end{displaymath}
	For a random vector $U$ in the unit sphere $\Sphere$ with distribution $\mu$, the
	uplifted process $\eta_c^U$ is also a Poisson process. Its intensity
	measure is given by
	\begin{equation}\label{eq:etaum}
	\E\;\eta_c^U(A \times [a,b])
	=\int_{u \in A} \int_a^b ct^{-1} dt \mu(du)=c\mu(A)\log (b/a)
	\end{equation}
	for all Borel $A \subseteq \Sphere$ and $0<a<b<\infty$.  It is
	immediately seen that $\eta_c^U$ is scale-invariant. By comparing
	the two equations above., we obtain that $\nu$ has the same intensity
	measure as $\eta_c^U$ with $c=\gamma_\nu(\Sphere)$ and $U$ is
	distributed according to $\mu=\gamma_\nu/c$. Thus $\nu=_d \eta_c^U$ proving
	the first claim.
	
	Next, for a random vector $V$ distributed on $S$ according to a
	probability measure $\psi$ with $\eta_c$ and $(V_k)_{k\in\N}$ satisfying \eqref{eq:condwd}, let $U=V/\|V\|$. Clearly, $\eta_c^V$ is
	also a Poisson process. For all $A \subseteq \Sphere$ and
	$0<a<b<\infty$, using the substitution $z=\|v\|t$ in the second
	step, the intensity of $\eta_c^V$ can be expressed as
	\begin{multline*}
	\E\;\eta_c^V(A\times[a,b])
	=\int_{v/\|v\|\in A, \|v\|t \in[a,b]} ct^{-1}dt\psi(dv)\\
	=\int_{v/\|v\|\in A, z \in[a,b]} cz^{-1}dz \psi(dv)=c \log (b/a)\Prob{U \in A}
	=\E\;\eta_c^U(A\times[a,b]),
	\end{multline*}
	where in the last step we have used \eqref{eq:etaum}. Hence,
	$\eta_c^V=_d\eta_c^U$.
\end{proof}

Recall that the generalized Dickman random variable $D_c$ with
parameter $c>0$ has the same distribution as the sum of points of
$\eta_c$ in the interval $(0,1)$. One can naturally generalize this
definition to dimensions $d\ge 2$ by considering a scale-invariant
Poisson process in $S$, which by Lemma~\ref{lem.char} is of the form
$\eta_c^U$ for some $c>0$ and unit random vector $U$ in $\R^d$, and
summing its points lying inside the unit ball $B_1$. The following
definition makes this precise.

\begin{definition}
	For a unit random vector $U$ in $\R^d$ and $c>0$, the
	\emph{multivariate Dickman} random variable $D_c^U$ with parameters $(c,U)$
	is defined by
	\begin{equation}\label{eq:GD}
	D_c^U= \int_{B_1}x \eta_c^U(dx)=\sum_{x \in \eta_c^U \cap B_1} x.
	\end{equation}
\end{definition}

Note that the points of $\eta_c$ in the interval $(0,1)$ are distributed
as the collection $\{Q_1^{1/c}, (Q_1 Q_2)^{1/c}, \dots \}$, where $(Q_k)_{k \in \N}$ are independent copies of a random variable $Q$ which is uniformly distributed on
$[0,1]$. Thus, letting $(U_i)_{i \in \N}$ be i.i.d.\ copies of $U$, we can
write
\begin{displaymath}
D_c^U=_d \sum_{x \in \eta_c^U} x \mathds{1}_{\{\|x\| < 1\}}
=_d \sum_{k=1}^\infty U_k \prod_{i=1}^k Q_i^{1/c}=_d Q^{1/c}(D_c^U+U')
\end{displaymath}
for $Q$  uniformly distributed on $[0,1]$ and $U'=_d U$ independent of $D_c^U$. Thus, the random
variable $D_c^U$ is the unique fixed point of the distributional
transformation $W \mapsto W^*$ given by
\begin{displaymath}
W^*=_d Q^{1/c}(W+U)
\end{displaymath}
with $Q$, $U$ and $W$ mutually independent.

By Lemma~\ref{lem.char}, the sum of points from any scale-invariant
Poisson process lying inside the unit ball is distributed as $D_c^U$
for some $c>0$ and unit random vector $U$. In particular, for a
general random vector $V$ in $S$, by Lemma~\ref{lem.char}, it is
straightforward to see that the sum of points of $\eta_c^V$ inside the
unit ball is distributed as $D_c^U$ with $U=V/\|V\|$.

Also note that
\begin{displaymath}
D_c^U=_d \sum_{k=1}^\infty e^{-Z_k}U_k,
\end{displaymath}
where $(U_k)_{k\geq1}$ are i.i.d.\ copies of $U$ and $(Z_k)_{k\in\N}$
is an enumeration of the points of a homogeneous Poisson process on the interval
$(0,1)$ with intensity $c$. In particular, $D_c^U$ is self-decomposable, see
\cite{sat80}.

We finish this section with an example of weak convergence to a
multivariate Dickman distribution as defined in
\eqref{eq:GD}. Consider the setting of Example~\ref{ex.mult} with
$d=2$ and $m=1$. Let $p_k$, $X_k$ and $V_k$ be as in
Example~\ref{ex.mult}.  For $p \in (0,1)$, let $X_k^1 \sim {\rm
	Bin}(X_k, p)\mathds{1}_{\{X_k>0\}}$ and $X_k^2=X_k-X_k^1$. Define
\begin{equation}\label{eq:7}
W_n=\sum_{k=1}^n X_k V_k \log p_k/\log p_n
=\frac{1}{\log p_n}\sum_{k=1}^n (X_k^1,X_k^2) \log p_k,
\end{equation} 
where $V_k=X_k^{-1}(X_k^1,X_k^2) \mathds{1}_{X_k>0}$. Let $D_1^{U}$ denote a Dickman random variable defined at
\eqref{eq:GD}, where $ U=(X,1-X)$ with $X \sim {\rm Ber}(p)$.

\begin{theorem}\label{2dimDickman}
	Let $W_n$ be given by \eqref{eq:7}. Then $W_n \vd D_1^{U}$ as $n \to
	\infty$.
\end{theorem}

\begin{proof}
	Define
	\begin{displaymath}
	\overline \nu_n= \sum_{k=\log n}^\infty Y_k \delta_{\log p_k/\log p_n},
	\end{displaymath}
	where the random variables $(Y_k)_{k\in\N}$ are independent with $Y_k\sim
	\mathrm{Ber}((1+p_k)^{-1})$ for $k\in\N$.  Notice that $\sum_{k=1}^{\log n-1} Y_k \delta_{\log p_k/\log p_n}$ converges vaguely in distribution to the zero process on $(0,\infty)$ as $n \to \infty$. By Remark~\ref{exp1}, the process
	$\sum_{k=1}^\infty Y_k \delta_{\log p_k/\log p_n}$ converges vaguely
	in distribution to $\eta$ as $n \to \infty$, hence, so does $(\overline \nu_n)_{n \in \N}$. By Theorem~\ref{thm4}, we
	obtain that $\overline \nu_n^U \vd \eta^U$ as $n \to \infty$.
	
	Let $E_n=\{X_k= \widetilde X_k \text{ for all } k \ge \log n\}$, where $\widetilde X_k=\mathds{1}_{\{X_k>0\}}$. Notice that for each $k$, the random variable $Y_k$ has the same law as $X_k$ conditional on the event $X_k=\widetilde X_k$. Hence, for each $n$, conditional on $E_n$, the point
	process $(X_k V_k \log p_k/\log p_n)_{\log n \le k \le n}$ has the
	same law as $\overline \nu_n^U$ restricted to the unit ball $B_1$.
	Therefore, the conditional law of 
	\begin{displaymath}
	Z_n=\sum_{k=\log n}^n X_k V_k \log p_k/\log p_n
	\end{displaymath}
	given $E_n$ is the same as that of $\int_{B_1} x d\overline \nu_n^{U}$.
	Using \cite[Prop.~1.51]{tao06}, notice that
	\[
	\lim_{\eps \to 0} \limsup_{n \to \infty} \E \int_{B_\eps} x \overline\nu_n^U(dx) \le \lim_{\eps \to 0} \limsup_{n \to \infty} \frac{1}{\log p_n}\sum_{k: p_k \le p_n^\eps}\frac{\log p_k}{1+p_k}=0.
	\]
	Since $\overline \nu_n^{U} \vd \eta^U$ as $n \to \infty$, using \cite[Theorem~4.28]{Ka02} and Lemma~\ref{conv.discont},
	it is not hard to see that
	\[
	(Z_n | E_n) =_d \int_{B_1} x \overline \nu_n^{U}(dx) \vd \int_{B_1} x \eta^{U}(dx)=_d D_1^U\; \text{ as }\; n \to \infty.
	\] 
	Since
	$\P(E_n)\to 1$ as $n \to \infty$, this yields that $Z_n \vd D_1^U$ as $n\to\infty$.  
	
	Finally, taking expectation and using \cite[Prop.~1.51]{tao06}, it is
	straightforward to see that
	\begin{displaymath}
	\sum_{k=1}^{\log n-1} X_k V_k \log p_k/\log p_n \to 0 
	\quad \text{ as }\; n \to \infty
	\end{displaymath}
	in $L^1$, hence, in probability as $n \to \infty.$ An
	application of Slutsky's theorem yields the result.
\end{proof}

\appendix
\section{Results on vague convergence in distribution}
\label{sec:appendix}
\renewcommand{\thesection}{A}

Let $S$ be a locally compact separable metric space. The following result provides a necessary and sufficient condition for
the vague convergence in distribution of a sequence of point processes
to a simple point process. Recall that a \textit{semi-ring}
$\mathcal{I}$ is a family of sets closed under finite intersections
such that any proper difference of sets in $\mathcal{I}$ is a finite,
disjoint union of $\mathcal{I}$-sets.

\begin{theorem}[see \protect{\cite[Theorem~4.15]{kalle17}}]
	\label{thm1} 
	Let $(\xi_n)_{n\geq1}$ be point processes on $S$, and fix a
	dissecting ring $\mathcal{U} \subset \widehat \fS_{\E \xi}$ and a
	semi-ring $\mathcal{I} \subset \mathcal{U}$. Then $\xi_n \vd \xi$ in
	$\sN(S)$ as $n \to \infty$ for a simple point process $\xi$ if and
	only if
	\begin{enumerate}[(i)]
		\item $\lim_{n \to \infty}\Prob{ \xi_{n} A = 0} =
		\Prob{ \xi A = 0 }$ for all $A \in \mathcal{U}$, and
		\item $\limsup_{n \to \infty} \Prob{\xi_{n} B > 1 } \leq
		\Prob { \xi B >1} $ for all $B \in \mathcal {I}$.
	\end{enumerate}
\end{theorem}

Recall that $\xi_n \vd \xi$ is equivalent to (see e.g.\ \cite[Theorerm~4.11]{kalle17})
\begin{equation}
\label{eq:2}
\int f(x)\xi_n(dx) \vd\int f(x) \xi(dx) \quad \text{as }\; n \to \infty
\end{equation}
for all $f\in\widehat{C}_S$. By a standard argument, approximating an
indicator function with a continuous function, it is straightforward to
derive the following result.

\begin{lemma}
	\label{conv.discont}
	Let $(\xi_n)_{n\geq1}, \xi$ be random measures in $S$ such that
	$\xi_n \vd \xi$ as $n \to \infty$. For a relatively compact
	measurable set $K$, let $f:S \to \R$ be a non-negative function
	which is continuous when restricted to $K$ and $f(x)=0$ for $x
	\notin K$. If $\E \xi(\partial K)=0$, then \eqref{eq:2} holds. 
\end{lemma}

Next we prove Lemma~\ref{lem1.1'}.

\begin{proof}[Proof of Lemma~\ref{lem1.1'}]
	Fix $\eps>0$ and $K_\eps$ satisfying \eqref{cond.sum'}. Since $\xi$ has a finite intensity, without loss of generality, we can assume that $\E\; \xi(\partial K_\eps)=0$. By
	Lemma~\ref{conv.discont},
	\begin{displaymath}
	\int_{K_\eps} h(x) \xi_n(dx) \vd
	\int_{K_\eps} h(x) \xi(dx) \quad
	\text{as }\; n\to\infty. 
	\end{displaymath}
	Hence,
	\begin{equation}
	\label{middle}
	\E\; \exp\left\{-\int_{K_\eps} h(x) \xi_n(dx)\right\}  
	\to \E\; \exp\left\{-\int_{K_\eps} h(x) \xi(dx)\right\}
	\end{equation}
	as $n\to\infty$.  Since $e^{\E X} \le \E\; e^X$,
	\begin{displaymath}
	\log\; \E\; \exp\left\{-\int_{K_\eps^c} h(x) \xi_n(dx)\right\}
	\ge-\E \int_{K_\eps^c} h(x) \xi_n(dx).
	\end{displaymath}
	Thus, by \eqref{cond.sum'}, we have that 
	\begin{equation}\label{below}
	\lim_{\eps \downarrow 0} \liminf_{n \to \infty} 
	\E\; \exp\left\{-\int_{K_\eps^c} h(x) \xi_n(dx)\right\}=e^0=1.
	\end{equation}
	The same holds for the upper limit trivially.  Combining \eqref{middle} and
	\eqref{below} yields the desired result.
\end{proof}

The following result which is a direct consequence of \cite[Theorem~4.28]{Ka02} and Lemma~\ref{conv.discont} provides conditions under which the vague
convergence of a general sequence of point processes
$(\nu_n)_{n\in\N}$ to $\eta_c$ implies the convergence of the sum of
points in $(0,1)$.

\begin{lemma}
	\label{lem:gensum} 
	Let $(\nu_n)_{n\in\N}$ be a sequence of point processes in
	$(0,\infty)$ with $\nu_n \vd \eta_c$ for $c>0$. If
	\begin{equation}
	\label{cond.sum}
	\lim_{\eps \to 0} \limsup_{n \to \infty} \E \int_0^\eps t \nu_n(dt)=0,
	\end{equation}
	then
	\begin{equation*}
	\int_0^1 t \nu_n(dt) \vd \int_0^1 t \eta_c(dt)=_d D_c.
	\end{equation*}
\end{lemma}

\begin{lemma}[Continuity of $M$ restricted to simple counting measures]\label{lem:contT}
	Let $(\xi_n)_{n\in\N}$ be a sequence of counting measures (deterministic) in
	$\sN((0,\infty))$ such that $\xi_n$ converges vaguely to $\xi$ as $n \to \infty$ for a
	simple counting measure $\xi$. If $M$ is given by \eqref{eq:T}, then $M(\xi_n)$ converges vaguely to $M(\xi)$ as $n \to \infty$.
\end{lemma}

\begin{proof}
	Denote $\widetilde \xi_n=M(\xi_n)$ and $\widetilde \xi=M(\xi)$. Note
	that it suffices to show that, for all $0<a<b<\infty$ and $k \in
	\N_0$, 
	\begin{equation}
	\label{eq:Tconv}
	\widetilde \xi_n([k,\infty) \times [a,b])
	\to \widetilde \xi([k,\infty) \times [a,b])\quad \text{as }\; n\to\infty.
	\end{equation}
	Since $\xi$ is simple, 
	\begin{displaymath}
	\widetilde \xi([k,\infty) \times [a,b])=\begin{cases}
	\xi([a,b]) \quad \text{for $k=0,1$},\\
	0 \qquad\quad\;\;\; \text{for $k>1$}.
	\end{cases}
	\end{displaymath}
	Note that $\widetilde \xi_n([k,\infty) \times [a,b])=\xi_n([a,b])$
	for $k=0,1$. Hence, \eqref{eq:Tconv} holds for $k=0,1$ by our
	assumption that $\xi_n \to \xi$ as $n \to \infty$. Fix $k>1$. Let
	$\xi([a,b])=m$ for some $m \ge 0$. If $m=0$, by our assumption we
	have $\xi_n([a,b]) \to \xi([a,b])=0$ as $n \to \infty$, which yields
	$\widetilde \xi_n([k,\infty) \times[a,b]) \to 0$ as $n \to \infty$,
	showing \eqref{eq:Tconv}. Next, assume that $m \ge 1$. Since $\xi$ is
	a locally finite counting measure, there are disjoint intervals
	$(I_i)_{1 \le i \le m}$ such that $\xi(I_i)=1$ for $1 \le i \le m$
	and $\cup_{i=1}^m I_i=[a,b]$. By our assumption, $\xi_n(I_i) \to
	\xi(I_i)=1$ as $n \to \infty$ for $1 \le i \le m$. Since $k>1$, we
	have $\widetilde \xi_n([k,\infty) \times I_i) \to 0$ as $n \to
	\infty$.  Taking union over the $m$ sets $[k,\infty) \times I_i$, $1
	\le i \le m$ proves \eqref{eq:Tconv}, concluding the proof.
\end{proof}

\section*{Acknowledgements}
We thank Matthias Schulte for a number of useful comments on the manuscript that greatly improved the paper. We would also like to thank Richard Arratia for pointing out the possible connection between convergence of sums and the convergence of underlying point processes.

\end{document}